\documentclass[12pt]{amsart}
\usepackage{amsmath,mathtools,amsthm,amsfonts,bigints,amssymb, mathrsfs, tikz-cd, xcolor} 

\newtheorem{theorem}{Theorem}[section]
\newtheorem{lemma}[theorem]{Lemma}

\theoremstyle{definition}
\newtheorem{definition}[theorem]{Definition}

\newtheorem{remark}[theorem]{Remark}

\newcommand{\R}{\mathbb R}%
\newcommand{\N}{\mathbb N}%
\numberwithin{equation}{section}
\makeatletter

\renewcommand\subsubsection{\@secnumfont}{\bfseries}%
\renewcommand\subsubsection{\@startsection{subsubsection}{3}
  \z@{.5\linespacing\@plus.7\linespacing}{-.5em}%
  {\normalfont\bfseries}}
  
  \makeatother

\makeatletter
\@namedef{subjclassname@2020}{%
  \textup{2020} Mathematics Subject Classification}
\makeatother 

\begin{document}

\title[Restricted Mean Value Property]{Restricted Mean Value Property with non-tangential boundary behavior on Riemannian manifolds}

\author[Utsav Dewan]{Utsav Dewan}
\address{Stat-Math Unit, Indian Statistical Institute, 203 B. T. Rd., Kolkata 700108, India}
\email{utsav\_r@isical.ac.in}



\begin{abstract} 
A well studied classical problem is the harmonicity of functions satisfying the restricted mean-value property (RMVP) for domains in $\mathbb{R}^n$. Recently, the author along with Biswas investigated the problem in the general setting of Riemannian manifolds and obtained results in terms of unrestricted boundary limits of the function on a full measure subset of the boundary. However in the context of classical Fatou-Littlewood type theorems for the boundary behavior of harmonic functions, a genuine query is to replace the condition on unrestricted boundary limits with the more natural notion of non-tangential boundary limits. The aim of this article is to answer this question in the local setup for pre-compact domains with smooth boundary in Riemannian manifolds and in the global setup for non-positively curved Harmonic manifolds of purely exponential volume growth. This extends a classical result of Fenton for the unit disk in $\mathbb{R}^2$.
\end{abstract} 

\maketitle
\tableofcontents

\section{Introduction}
We recall the following characterization of harmonic functions in Euclidean domains: a continuous function in a domain $\Omega$ in $\R^n$ is harmonic if and only if it satisfies the spherical mean-value property (with respect to the surface volume measure) for all spheres contained in $\Omega$. Then a natural question to ask is if one instead assumes a much weaker version of the mean-value property, does harmonicity still hold true. Namely, a continuous function on a domain $\Omega$ in $\mathbb{R}^n$ is said to satisfy the {\it restricted mean-value property} (RMVP for short) if for each point $x \in \Omega$, there exists a sphere $S_x$ with center $x$ and some radius $\rho(x)$, which is contained in the domain $\Omega$, such that $u(x)$ equals the mean-value of $u$ on the sphere $S_x$. One can then ask whether a function satisfying the restricted mean-value property in $\Omega$ is harmonic in $\Omega$.

\medskip

For a bounded domain $\Omega$, the boundary behavior of the function seems to play an important role in answering such a question. Indeed, a simple argument of Kellogg shows that if the function $u$ extends continuously to the closure of the domain $\Omega$, then $u$ must be harmonic (\cite{Kellogg}). Without any assumptions on the boundary behavior however, there are counter-examples (see page 22, \cite{Littlewood}). These counter-examples are of specific nature: they are unbounded. In this context, Littlewood asked (in \cite{Littlewood}) for $n=2$, whether the unboundedness is the only obstruction, more precisely, if the function satisfying the RMVP is also assumed to be bounded, then is it harmonic? This is the classical `one-circle problem' and in \cite{HN2}, Hansen and Nadirashvili showed that the above problem has a negative answer, that is, there exists a continuous, bounded function on the unit disk in $\R^2$ which satisfies RMVP but is not harmonic. The problem is still open for $n \ge 3$ however. Nevertheless, obtaining sufficient conditions for a function satisfying RMVP to be harmonic has received considerable attention over the years, 
in particular from Fenton and Hansen-Nadirashvili (\cite{Fenton,Fenton2,Fenton3,HN1}).

\medskip

While the above articles addressed this problem for domains in $\mathbb{R}^n$, the author along with Biswas has recently investigated the problem in the general setting of Riemannian manifolds \cite{BD}. Now it is not true in general that harmonic functions on Riemannian manifolds satisfy the mean-value property when the mean-values are taken over geodesic spheres with respect to the surface volume measure induced by the Riemannian metric. However the mean-value property does hold true when one considers the mean-values with respect to the harmonic measures instead. Thus the results in \cite{BD} require a reformulation of the RMVP in terms of spherical mean-values over geodesic spheres with respect to harmonic measures, which we restate here for the sake of completeness:
\begin{definition}\cite[definition 1.1]{BD} \label{defn_rmvp}
Let $\Omega$ be a domain in a Riemannian manifold $M$. A continuous function $u$ on $\Omega$ is said to satisfy the {\it Restricted Mean Value Property} in $\Omega$ if for all $z \in \Omega$, there exists $0<\rho(z)<inj(z)$ (where $inj(z)$ is the injectivity radius of $z$) such that the closed ball $\overline{B(z,\rho(z))}$ is contained in $\Omega$, and one has the equality,
\begin{equation*} 
u(z) = \int_{\partial B(z,\rho(z))} u(\xi) \:d\mu_{z,B(z,\rho(z))}(\xi) \:,
\end{equation*}
where $\mu_{z,B(z,\rho(z))}$ is the harmonic measure on $\partial B(z,\rho(z))$ with respect to $z$. 
\end{definition}

The main results (Theorems $1.2$ and $1.4$) in \cite{BD} are remarkable generalizations of a classical result of Fenton (for the unit disk in $\R^2$) \cite[Theorem 2]{Fenton} with the general outline: if a continuous, bounded, real-valued function satisfying the RMVP has unrestricted boundary limits on a full measure subset of the boundary, then it is harmonic. Now in the classical setting of the unit disk $\mathbb D$ in $\R^2$, the celebrated result of Fatou \cite{Fatou} asserts that bounded harmonic functions in $\mathbb D$ have non-tangential limits almost everywhere on the unit circle. In fact, the non-tangential approach regions were shown to be sharp by Littlewood in \cite{Littlewood1}, where he constructed a bounded harmonic function in $\mathbb D$, that fails to have tangential limits almost everywhere on the unit circle. This result was further improved by Aikawa in \cite{Aikawa}, where he constructed a bounded harmonic function in $\mathbb D$, that fails to have tangential limits at all points on the unit circle. Then in the context of the above results on the boundary behavior of harmonic functions, it becomes a genuine query to replace the condition on unrestricted boundary limits (for the results in \cite{BD}) with the more natural notion of non-tangential boundary limits. The aim of this article is to address this question. 

\medskip

The only known result in this direction is for the unit disk $\mathbb D$ in $\R^2$, and is due to Fenton \cite[Theorem 3]{Fenton}: let $\phi$ be such that $0<\phi<\pi/2$, and let $u$ be a bounded, continuous, real-valued function which is such that 
\begin{equation*}
\displaystyle\lim_{z \to e^{i \theta}} u(z) \text{ exists for almost all } \theta\:,
\end{equation*}
whenever $z$ approaches to $e^{i \theta}$ through a sector of opening $2\phi$ with vertex $e^{i \theta}$, symmetrically disposed about the normal to the unit circle at $e^{i \theta}$. If $u$ has the RMVP on circles of radius
\begin{equation*}
\rho(\zeta) \le (1-|\zeta|) \tan (\phi/2)\:,
\end{equation*}
then $u$ is harmonic in $\mathbb D$.

\medskip

We note that in the above result, the function is assumed to have boundary limits only through a non-tangential cone of fixed apperture (with vertices belonging to a full measure subset of the boundary). This is a much weaker hypothesis than the existence of non-tangential limits which assumes the above to be true for any non-tangential cone.

\medskip

We first see a generalization of the above result for domains in a general Riemannian manifold. Henceforth $M$ will denote a Riemannian manifold, and $\Omega \subset M$ will be a precompact domain in $M$ with smooth boundary, such that the distance to the boundary $d(z,\partial \Omega)$ is smaller than the injectivity radius $inj(z)$, for all points $z$ in $\Omega$. Following \cite{Kenig}, next we define analogues of non-tangential cones in this general setup.

\begin{definition} \label{non-tangential_cone_domain}
Let $\Omega$ be a smooth, pre-compact domain in a Riemannian manifold $M$. For $\xi \in \partial \Omega$ and $\alpha \in (0,\infty)$, the non-tangential cones are defined as,
\begin{equation*}
\Gamma_\alpha(\xi):=\left\{p \in \Omega: d(\xi,p) \le (1+\alpha)d(p,\partial \Omega)\right\} \:.
\end{equation*}
\end{definition}

We note that by the smoothness of the boundary $\partial \Omega$, given $\alpha \in (0,\infty)$ and $\xi \in \partial \Omega$, there exists $\varepsilon>0$, such that 
\begin{equation*}
\overline{\Gamma_\alpha(\xi)} \cap B(\xi,\varepsilon) \cap \partial \Omega = \{\xi\}\:.
\end{equation*}

\medskip

In this setting we have the following result: 
\begin{theorem} \label{domain_nt}
For a fixed $\alpha \in (0,\infty)$, we fix a constant $\kappa \in \left(0,\frac{\alpha}{4+\alpha}\right)$. If $u: \Omega \subset M \to \R$ is bounded, continuous, satisfies the Restricted Mean Value Property at each point $z \in \Omega$ on a geodesic sphere of radius $\rho(z) \le \kappa d(z,\partial \Omega)$,  and 
\begin{equation*}  
\displaystyle\lim_{z \to \xi} u(z) = u_\xi \:\text{exists}  \:,
\end{equation*}
whenever $z$ approaches to $\xi$ through the non-tangential cone $\Gamma_\alpha(\xi)$, for almost every $\xi \in \partial \Omega$ with respect to the Riemannian measure on $\partial \Omega$, then  $u$ is harmonic in $\Omega$.
\end{theorem}

We also have a result for when the boundary is ``at infinity". Namely, if we take $X$ to be a non-positively curved Harmonic manifold of purely exponential volume growth, then it has infinite injectivity radius at each point, and we can consider functions on the whole manifold $X$ satisfying definition \ref{defn_rmvp}, where $X$ itself is the domain under consideration. In this case however, one needs to consider the {\it Gromov boundary} $\partial X$ of the Gromov hyperbolic space $X$, and the appropriate measure class on the Gromov boundary in this context is the class of visibility measures. We refer the reader to section $2$ for the definition of the Gromov boundary and the visibility measures on the Gromov boundary. 

\medskip

Let us briefly recall that a Harmonic manifold is a complete, simply connected Riemannian manifold $X$ such that for any point $x \in X$, there exists a non-constant harmonic function on a punctured neighbourhood of $x$ which is radial around $x$, that is, depends only on the geodesic distance from $x$. By purely exponential volume growth, we mean that there exist constants $C > 1,\: h > 0$, such that the volume of metric balls $B(x, R)$ with center $x \in X$ and radius $R>1$, satisfies the asymptotics:
\begin{equation*}
\frac{1}{C} e^{hR} \le vol(B(x,R)) \le C e^{hR} \:.
\end{equation*}

The class of non-positively curved Harmonic manifolds of purely exponential volume growth includes all the known examples of non-compact non-flat Harmonic manifolds: the rank one Riemannian symmetric spaces of non-compact type and the Damek-Ricci spaces.

\medskip

In order to state our result, we recall the standard notion of non-tangential cones defined in terms of the intrinsic geometry of $X$. 
\begin{definition} \label{non-tangential_cone_full_space}
Fix an origin $o \in X$. For $\xi \in \partial X$, let $\gamma_\xi$ denote the unit-speed geodesic ray such that $\gamma_\xi(0)=o,\:\gamma_\xi(+\infty)=\xi$ and $d(o,\gamma_\xi(t))=t$. For $C>0$, the non-tangential cone, $T_C(\xi)$ is defined by
\begin{equation*}
T_C(\xi):=\{x \in X: d(x,\gamma_\xi) \le C\}\:.
\end{equation*}
\end{definition}

\medskip

We then have:
\begin{theorem} \label{harmonic_manifold_nt}
Fix a constant $C>0$, and let $X$ be a non-positively curved Harmonic manifold of purely exponential volume growth with origin $o$. If $u:X \to \R$ is  bounded, continuous, satisfies the Restricted Mean Value Property at each point $z \in X$ on a geodesic sphere of radius $\rho(z) \le C/2$ and 
\begin{equation*} 
\displaystyle\lim_{z \to \xi} u(z) = u_\xi \:\text{exists} \:,
\end{equation*}
whenever $z$ approaches to $\xi$ through the non-tangential cone $T_C(\xi)$, for almost every $\xi \in \partial X$ with respect to the visibility measure $\lambda_o$ on $\partial X$, then $u$ is harmonic in $X$.
\end{theorem}
Both the results above follow the general scheme of the arguments in \cite{Fenton}, 
namely constructing subharmonic and superharmonic functions $v$ and $w$ respectively such that $w \leq u \leq v$, and then showing that $v-w$ vanishes identically.

\medskip

But the implementation in our general setting has some subtleties and require new technical considerations:

\begin{enumerate}
\item By using the condition on the radius function of RMVP suitably, we reduce the first step to a framework so that the machineries developed in \cite{BD} (concentration of harmonic measures and convergence of Poisson integrals) can be applied. This is done in section $3$.

\medskip

\item For the final step, we carry out a uniqueness argument for a suitable subharmonic function in terms of its normal limits or limits along radial geodesic rays terminating on a full measure subset of the boundary. More precisely, 
\begin{itemize}
\item For smooth, pre-compact domains in Riemannian manifolds, while only assuming that a bounded, real-valued subharmonic function has normal limits (see definition \ref{normal_convergence}) on a full measure subset of the boundary, we show that it admits a harmonic majorant given by the Poisson intgeral of its normal limits, that is, by the average of its normal limits against the harmonic measures (see Lemma \ref{max_princ}). This is done by working out asymptotic ratios of Poisson kernels under tubular perturbation of domains. In the absence of an explicit expression of the Poisson kernel on general domains, this is somewhat non-trivial. This is done in section $4$.
\item  In the global setup of Harmonic manifolds, we show that if a bounded subharmonic function has vanishing limits only along radial geodesic rays emanating from a fixed origin and terminating on a full measure subset of the Gromov boundary $\partial X$, then the least harmonic majorant of that subharmonic function has a zero at the origin. This is done by using an explicit expression of the least harmonic majorant in terms of a limit of spherical averages of the  subharmonic function under consideration (see Lemma \ref{least_harmonic_majorant}), which is a very recent development in the area of potential theory on non-positively curved Harmonic manifolds of purely exponential volume growth, made by the author \cite{D}. 
\end{itemize}
\end{enumerate}   

\medskip

This article is organized as follows. In section $2$, we recall some definitions and fix our notations. In section $3$, we construct two important auxiliary functions and discuss some of their crucial properties. In section $4$, we establish an $L^\infty$ maximum principle for subharmonic functions. In section $5$, we prove Theorems \ref{domain_nt} and \ref{harmonic_manifold_nt}. 

\section{Preliminaries}
Let $M$ be a Riemannian manifold. For a precompact domain $\Omega$ with smooth boundary in $M$, solving the Dirichlet problem gives rise to a family of probability measures $\{\mu_{x,\Omega}\}_{x \in \Omega}$ on $\partial \Omega$. For any $x \in \Omega$, the {\it harmonic measure on $\partial \Omega$ with respect to $x$}, $\mu_{x,\Omega}$ is defined by,
\begin{equation*}
\int_{\partial \Omega} f \: d\mu_{x,\Omega} = u_f(x) \:,
\end{equation*} 
for all continuous functions $f$ on $\partial \Omega$, where $u_f$ is the unique solution of the Dirichlet problem in $\Omega$ with boundary value $f$. The harmonic measures $\mu_{x,\Omega}$ are mutually absolutely continuous, in fact they are absolutely continuous with respect to the Riemannian measure on $\partial \Omega$. 

\medskip

For every $x \in M$, the Riemannian exponential map at $x$, $exp_x : T_xM \to M$ is a local diffeomorphism. For $x \in M$, we look at all the balls  centered at $0\in T_xM$ with radius $r>0$, $B(0,r)$ in $T_xM$ such that $exp_x|_{B(0,r)}$ is a diffeomorphism onto its image. The supremum all such radii is the {\it injectivity radius of $x$}, denoted by $inj(x)$. All metric spheres centered at $x$ with radius smaller than $inj(x)$ are geodesic spheres. Hence they are $C^\infty$-submanifolds of $M$ and are equipped with the family of harmonic measures.

\medskip

The harmonic measures allow us to define the {\it Poisson integral} of any $L^{\infty}$ function $f \in L^{\infty}(\partial \Omega)$, which is a bounded harmonic function on $\Omega$ defined by
\begin{equation*}
P[f](x) := \int_{\partial \Omega} f \ d\mu_{x, \Omega} \ , \ x \in \Omega.
\end{equation*}

\medskip

Let $\Delta$ denote the Laplace-Beltrami operator on $M$ corresponding to the underlying Riemannian metric. For any $x \in \Omega$, one denotes by $G^x_{\Omega}$ the corresponding Green function. It is the unique function on $\overline{\Omega}$ which is continuous outside $x$ and such that
\begin{equation*}
\Delta G^x_{\Omega}=-\delta_x  \text{ in } \Omega \ \text{ and } G^x_{\Omega} \equiv 0  \text{ on } \partial \Omega \ .  
\end{equation*}
For $\xi \in \partial \Omega$, the {\it Poisson kernel of $\Omega$} is defined as the normal derivative of $G^x$ along the unit inward normal vector to $\partial \Omega$ at $\xi$,
\begin{equation*}
P_{\Omega}(x,\xi):= \frac{\partial G^x_{\Omega}}{\partial n}(\xi) \:.
\end{equation*}
$P_{\Omega}(\cdot,\xi)$ is a positive harmonic function in $\Omega$, for all $\xi \in \partial \Omega$. The Poisson kernel relates the harmonic measure $\mu_{x,\Omega}$, with the normalized surface volume measure on $\partial \Omega$ (induced by the Riemannian metric) $\Sigma_{\partial \Omega}$ in the following way:
\begin{equation*}
\mu_{x,\Omega}=P_{\Omega}(x,\cdot)\ \Sigma_{\partial \Omega}\ .
\end{equation*}

Next we recall briefly some basic facts about Gromov hyperbolic spaces. These can be found in \cite{Bridson}. 

\medskip

A {\it geodesic} in a metric space $X$ is an isometric embedding $\gamma : I \subset \mathbb{R} \to X$ of an interval into $X$. The metric space $X$ is called {\it geodesic} if any two points in $X$ can be joined by a geodesic. A geodesic metric space $X$ is called {\it Gromov hyperbolic} if there is a $\delta \geq 0$, such that for every geodesic triangle in $X$, each side is contained in the $\delta$-neighbourhood of the union of the other two sides.

\medskip

The {\it Gromov boundary} of a Gromov hyperbolic space $X$ is defined to be the set $\partial X$ consisting of equivalence classes of geodesic rays in $X$. A geodesic ray is an isometric embedding $\gamma : [0,\infty) \to X$ of a closed half-line into $X$, and two geodesic rays $\gamma_1, \gamma_2$ are called equivalent if the set $\{ d(\gamma_1(t), \gamma_2(t)) \ | \ t \geq 0 \}$ is bounded. The 
equivalence class of a geodesic ray $\gamma$ is denoted by $\gamma(\infty) \in \partial X$. 

\medskip

A metric space is said to be {\it proper} if closed and bounded balls in the space 
are compact. Let $X$ be a proper, geodesic, Gromov hyperbolic space. There is a natural topology on $\overline{X} := X \cup \partial X$, called the {\it cone topology} such that $\overline{X}$ is a compact metrizable space which is a compactification of $X$. In this case, for every geodesic ray $\gamma$,  $\gamma(t) \to \gamma(\infty) \in \partial X$ as $t \to \infty$, and for any $x \in X,\ \xi \in \partial X$ there exists a geodesic ray  
$\gamma$ such that $\gamma(0) = x, \gamma(\infty) = \xi$. 

\medskip

Now let us recall some basic properties of a non-positively curved Harmonic manifold of purely exponential volume growth $X$ defined in the Introduction. On such a manifold, the harmonic functions satisfy the usual mean value property on spheres with respect to the surface volume measure induced by the Riemannian metric \cite{Willmore}.

\medskip

In \cite{Kn12}, Knieper showed that for $X$, a simply connected non-compact Harmonic manifold of purely exponential
volume growth with respect to a fixed basepoint $o \in X$, the condition of purely exponential volume growth is equivalent to either of the following conditions:
\begin{enumerate}
\item $X$ is Gromov hyperbolic.
\item $X$ has rank one.
\item The geodesic flow of $X$ is Anosov with respect to the Sasaki metric.
\end{enumerate}

\medskip

Moreover, the Gromov boundary coincides with the visibility boundary $\partial X$ introduced in \cite{Eberlein}. One has a family of measures on $\partial X$ called the visibility measures $\{\lambda_x\}_{x \in X}$. For $x \in X$, let $\theta_x$ denote the normalized canonical measure on $T^1_x X$ (the unit tangent space at $x$), induced by the Riemannian metric and then the visibility measure $\lambda_x$ is obtained as the push-forward of $\theta_x$ to the boundary $\partial X$ under the radial projection. The visibility measures $\lambda_x$ are pairwise absolutely continuous. This can be found in \cite{KP16}.

\medskip

A subharmonic function $f$ on $X$ is said to have a harmonic majorant if there exists a harmonic function $h$ on $X$ such that 
\begin{equation*}
f(x) \le h(x)\:,\text{ for all } x \in X\:.
\end{equation*}
A harmonic function $h$ on $X$ is called the {\it least harmonic majorant} of a subharmonic function $f$ on $X$ if
\begin{itemize}
\item $h$ is a harmonic majorant of $f$ and
\item $h(x) \le H(x)$ for all $x \in X$, whenever $H$ is a harmonic majorant of $f$\:.
\end{itemize} 
For $x \in X$ and $v \in T^1_x X$, let $\gamma_{x,v}$ denote the geodesic such that $\gamma_{x,v}(0)=x$ and $\gamma'_{x,v}(0)=v$.
The following result concerning an explicit expression of the least harmonic majorant of a subharmonic function follows from the proof of Proposition $4.10$ in \cite{D}: 
\begin{lemma} \label{least_harmonic_majorant}
Let $f$ be a subharmonic function having a harmonic majorant on $X$, then $f$ has a least harmonic manjorant $F_f$ and it is given by,
\begin{equation*}
F_f(x)= \displaystyle\lim_{r \to \infty} \int_{T^1_x X} f\left(\gamma_{x,v}(r)\right)\:d\theta_x(v) \ , \text{ for all } x \in X\ .
\end{equation*}
\end{lemma}

\section{Construction of two auxiliary functions}
In this section, we construct a subharmonic function and a superharmonic function that will play a pivotal role in the proofs of our main theorems.

\medskip

Let $D$ be either $\Omega$ (as in Theorem \ref{domain_nt}) or $X$ (as in Theorem   \ref{harmonic_manifold_nt}). For $z \in D$, 
\begin{itemize}
\item under the hypothesis of Theorem \ref{domain_nt}, we define $\rho_0(z)=\kappa d(z, \partial D)$,
\item under the hypothesis of Theorem \ref{harmonic_manifold_nt}, we define $\rho_0(z) \equiv C/2\:.$
\end{itemize}

\medskip

Now let $\mathscr{F}_z$ be the collection of 
harmonic extensions of $u$ on balls $B(x,r)$ where $x \in D, r \le \rho_0(x), z \in B(x,r)$ and we have 
\begin{equation*}
u(x)=\int_{\partial B(x,r)} u(y)\:d\mu_{x,B(x,r)}(y) \:.
\end{equation*}

\medskip

We now define,
\begin{equation} \label{defn_v}
v(z):= \sup_{h \in \mathscr{F}_z} h(z) \:, \text{ for all } z \in D.
\end{equation}

\medskip

Then note that, since $u$ satisfies the RMVP, for any $z \in D$, there exists $h \in \mathscr{F}_z$ such that $u(z) = h(z)$, and hence  
\begin{equation} \label{u<v}
u(z) \le v(z) \:, \text{ for all } z \in D.
\end{equation}
The next lemma summarizes some important properties of $v$.
\begin{lemma} \label{props_of_v}
For all $z \in D$, there exists $h \in \mathscr{F}_z$ such that $v(z)=h(z)$. Moreover, $v$ is bounded, continuous and subharmonic in $D$. 
\end{lemma}
\begin{proof}
First we show that the supremum in the definition (\ref{defn_v}) is attained. Choose and fix $z_0 \in D$. Let $\{h_n\}_{n=1}^\infty \in \mathscr{F}_{z_0}$ such that 
\begin{equation*}
h_n(z_0) \to v(z_0) \: \text{ as } n \to \infty \:.
\end{equation*}
For each $n \in \N$, let $B_n$ be the ball such that $h_n$ is the harmonic extension of $u$ on $B_n$.

\medskip

Thanks to the conditions on $\rho_0$, passing to a subsequence, we may assume that $\overline{B_n} \to \overline{B}$ and $\partial B_n \to \partial B$ (for some ball $B=B(x,r) \subset D$ with $r \le \rho_0(x)$) in the Hausdorff topology on compacts in $\overline{D}$. As $z_0 \in B_n$ for all $n \in \N$, there are three cases: 
\begin{enumerate}
\item[(i)] $B=\{z_0\}$\:,
\item[(ii)] $B$ is non-degenerate and $z_0$ lies in the interior of $B$,
\item[(iii)] $B$ is non-degenerate and $z_0 \in \partial B$.
\end{enumerate}

\medskip

In the setup of Theorem \ref{domain_nt}, the condition on $\rho_0$ ensures that there exists a domain $D'$ compactly contained in $\Omega$ such that all the balls $B_n$ and $B$ are contained in $D'$. In the setup of Theorem \ref{harmonic_manifold_nt}, as $z_0 \in B_n$ for all $n \in \N$, by the condition on $\rho_0$, we see that all the balls $B_n$ and $B$ are contained in the ball $B(z_0,C)$. Thus in both the cases, there exists a domain $D'$ compactly contained in $D$ such that all the balls $B_n$ and $B$ are contained in $D'$. Then by continuity of $u$ on $\overline{D'} \subset D$ and proceeding as in the scenario of a pre-compact domain in a Riemannian manifold in the proof of Lemma $4.1$ in \cite{BD}, we get that in all the above cases, there exists $h \in \mathscr{F}_{z_0}$ such that $v(z_0)=h(z_0)$. This completes the proof of the first part of the Lemma.

\medskip

To show continuity of $v$, we again choose and fix $z_0 \in D$. By the first part of this Lemma, there exists $h \in \mathscr{F}_{z_0}$ such that $v(z_0)=h(z_0)$. Let $B$ be the ball in $D$ containing $z_0$ such that $h$ is the harmonic extension of $u$ on $B$. We note that $h \in \mathscr{F}_z$, for all $z \in B$. Then by definition of $v$ (\ref{defn_v}), we have
\begin{equation} \label{v>h}
v(z) \ge h(z) \ , \text{ for all } z \in B.
\end{equation}
Hence by continuity of $h$,
\begin{equation} \label{liminf}
\liminf_{z \to z_0} v(z) \ge \liminf_{z \to z_0} h(z) = h(z_0) = v(z_0)\ .
\end{equation}

\medskip

Now for the limsup, we consider $\{z_n\}_{n=1}^\infty \subset D$ such that $z_n \to z_0$ as $n \to \infty$. Let $h_n \in \mathscr{F}_{z_n}$ be the harmonic extensions of $u$ on balls $B_n \subset D$ containing $z_n$ such that $v(z_n)=h_n(z_n)$. Thanks to the conditions on $\rho_0$, passing to a subsequence, we may assume that $\overline{B_n} \to \overline{B}$ and $\partial B_n \to \partial B$ (for some ball $B=B(x,r) \subset D$ with $r \le \rho_0(x)$) in the Hausdorff topology on compacts in $\overline{D}$. Again there are three cases: 
\begin{enumerate}
\item[(i)] $B=\{z_0\}$\:,
\item[(ii)] $B$ is non-degenerate and $z_0$ lies in the interior of $B$,
\item[(iii)] $B$ is non-degenerate and $z_0 \in \partial B$.
\end{enumerate}
Then the conditions on $\rho_0$ ensures that there exists a domain $D'$ compactly contained in $D$ such that all the balls $B_n$ and $B$ are contained in $D'$. Then using continuity of u on $\overline{D'} \subset D$ and proceeding as in the scenario of a pre-compact domain in a Riemannian manifold in the proof of Lemma $4.2$ in \cite{BD}, combining all the above cases, we get that
\begin{equation} \label{limsup}
\limsup_{z \to z_0} v(z) \le v(z_0)\:.
\end{equation}
Thus combining (\ref{liminf}) and (\ref{limsup}), we get the continuity of $v$.

\medskip

To see that $v$ is subharmonic in $D$, again we choose and fix $z_0 \in D$. Let $h \in \mathscr{F}_{z_0}$ be the harmonic extension of $u$ on a ball $B \subset D$ containing $z_0$ such that $v(z_0)=h(z_0)$. Then for all balls $B'$ centered at $z_0$ and compactly contained in $B$, by (\ref{v>h}), it follows that
\begin{equation*}
v(z_0)=h(z_0)=\int_{\partial B'} h(y)\: d\mu_{z_0,B'}(y) \le \int_{\partial B'} v(y)\: d\mu_{z_0,B'}(y) \:.
\end{equation*}
Thus $v$ is subharmonic at $z_0$. Since $z_0 \in D$ was arbitrarily chosen and subharmonicity is a local property, it follows that $v$ is subharmonic in $D$. 

\medskip

Boundedness of $v$ is a simple consequence of the facts that $u$ is bounded, (\ref{u<v}), the definition of $v$ (\ref{defn_v}) and the maximum principle.
\end{proof}

\begin{remark} \label{props_of_w}
One may also define $w(z):= \displaystyle\inf_{h \in \mathscr{F}_z} h(z)$, for all $z \in D$. Proceeding similarly as above it can be shown that 
\begin{itemize}
\item for all $z \in D$, there exists $h \in \mathscr{F}_z$ such that $w(z)=h(z)$,
\item $w$ is bounded, continuous, superharmonic in $D$ satisfying $w(z) \le u(z)$, for all $z \in D$ and hence
\begin{equation*}
w \le u \le v \:.
\end{equation*}
\end{itemize}
\end{remark}

\section{An $L^\infty$ maximum principle}
We start off this section by defining the notion of normal convergence at a boundary point of a smooth pre-compact domain.
\begin{definition} \label{normal_convergence}
Let $\Omega$ be a smooth pre-compact domain in a Riemannian manifold $M$. Then for $\xi \in \partial \Omega$, a sequence $\{z_n\}_{n=1}^\infty \subset \Omega$ is said to {\em converge normally to $\xi$} if $z_n \to \xi$ with
\begin{equation*}
\frac{d(z_n,\xi)}{d(z_n,\partial \Omega)} \to 1 \text{ as } n \to \infty\:. 
\end{equation*}
\end{definition}

Next we see the main result of this section, a pointwise maximum principle based on only the normal limits of a subharmonic function on a full measure subset of the boundary:  
\begin{lemma} \label{max_princ}
Let $\varphi$ be a bounded, real-valued subharmonic function on a smooth, pre-compact domain $\Omega$ in a Riemannian manifold $M$ such that  
\begin{equation*}
\displaystyle\lim_{z \to \xi} \varphi(z) = f(\xi) \:,
\end{equation*}
whenever $z$ converges normally to $\xi$, for almost every boundary point $\xi$, with respect to the harmonic measures on $\partial \Omega$, where $f \in L^\infty(\partial \Omega)$. Then  
\begin{equation*}
\varphi(x) \le P[f](x) \:,
\end{equation*}
for all $x \in \Omega$, where $P[f]$ is the Poisson integral of $f$.
\end{lemma}

\begin{remark} \label{comparison_max_princ}
Lemma \ref{max_princ} implies that a suitable subharmonic function on a pre-compact domain admits a harmonic majorant given by the Poisson integral of its normal boundary limits. We note that if a function has non-tangential limit or unrestricted limit, say $L$, at a boundary point $\xi$, then in particular, it has normal limit $L$ at $\xi$. Hence from Lemma \ref{max_princ}, one can conclude similar $L^\infty$ maximum principles for bounded, continuous subharmonic functions that have non-tangential or unrestricted limits almost everywhere on the boundary. In particular, Lemma $5.1$ of \cite{BD}, which is concerned with unrestricted limits of a subharmonic function, is a corollary of Lemma \ref{max_princ}.    
\end{remark}

For proving Lemma \ref{max_princ}, we will need the following Lemmas regarding the Poisson kernel. 
\begin{lemma} \label{Poisson_estimate_lemma}
Let $\Omega$ be a smooth, pre-compact domain in a Riemannian manifold $M$. Let $\delta>0$, be such that $\Omega_{\delta}:=\{x \in \Omega: d(x,\partial \Omega) \ge \delta\}$ is non-empty. Then there exists a constant $C \ge 1$, which continuously depends only on $\Omega$ and $\delta$, such that for every $\xi \in \partial \Omega$ and every $x \in \Omega_{\delta}$, the Poisson kernel satisfies
\begin{equation} \label{Poisson_estimate} 
\frac{1}{C} \le P_\Omega(x,\xi) \le C \:.
\end{equation}
\end{lemma}
\begin{proof}
As $P_\Omega(\cdot,\cdot)$ is a positive continuous function on the compact set  $\Omega_\delta \times \partial \Omega$, there exist positive constants $C_1,C_2$ continuously depending only on $\Omega$ and $\delta$ such that for any $x \in \Omega_{\delta}$ and any $\xi \in \partial \Omega$,
\begin{equation} \label{Poisson_estimate_eq1}
C_1 \le P_\Omega(x,\xi) \le C_2 \ .
\end{equation}

\medskip

The estimate (\ref{Poisson_estimate}) now follows from (\ref{Poisson_estimate_eq1}) by setting $C=\max\{C_2,\ 1/C_1\}$.
\end{proof}
A consequence of Lemma \ref{Poisson_estimate_lemma}, is the following asymptotic of ratios of Poisson kernels under perturbation of domains:
\begin{lemma}\label{Poisson_asymptotic_lemma}
Let $\Omega$ be a smooth, pre-compact domain in a Riemannian manifold $M$ and $x \in \Omega$. For $\varepsilon>0$ sufficiently small, let $\Omega_{\varepsilon}:=\{y \in \Omega : d(y,\partial \Omega)> \varepsilon\}$ be a non-empty domain such that 
\begin{itemize}
\item $\overline{\Omega} \setminus \overline{\Omega_{\varepsilon}}$ is a tubular neighborhood of $\partial \Omega$,
\item $x \in \Omega_\varepsilon$ with $d(x,\partial \Omega_\varepsilon) > \frac{1}{2} d(x,\partial \Omega)$.
\end{itemize}
Let $P_\Omega$ and $P_{\Omega_\varepsilon}$ denote the Poisson kernels of $\Omega$ and $\Omega_{\varepsilon}$ respectively. Let $\psi_{\varepsilon}$ denote the radial projection map from $\partial \Omega$ onto $\partial \Omega_{\varepsilon}$. Then for all $\xi \in \partial \Omega$, 
\begin{equation*}  
\frac{P_{\Omega_\varepsilon}(x,\psi_{\varepsilon}(\xi))}{P_\Omega(x,\xi)} = 1+o(1) \ \text{ as } \varepsilon \to 0 \ .
\end{equation*}
\end{lemma}
\begin{proof}
For $\delta=d(x,\partial \Omega)$, let $C\ge 1$, be as in the conclusion of Lemma \ref{Poisson_estimate_lemma}. Then by $C^1$-dependence of the Green functions on their defining domains (see for instance, \cite[p. 311, (2.17)]{Ito}) and the fact that $\psi_{\varepsilon} \to id_{\partial \Omega}$ in $C^\infty$ as $\varepsilon \to 0$, given $\eta >0$, there exists $\varepsilon_0 >0$, such that for all $\varepsilon \in (0,\varepsilon_0)$ and all $\xi \in \partial \Omega$, one has
\begin{equation} \label{Poisson_asymptotic_eq1}
\left|P_{\Omega_\varepsilon}(x,\psi_{\varepsilon}(\xi)) - P_\Omega(x,\xi) \right| < \frac{\eta}{C} \ .
\end{equation} 
Then combining (\ref{Poisson_asymptotic_eq1}) and (\ref{Poisson_estimate}), we get that given $\eta >0$, there exists $\varepsilon_0 >0$, such that for all $\varepsilon \in (0,\varepsilon_0)$ and all $\xi \in \partial \Omega$, 
\begin{equation*}
\left|\frac{P_{\Omega_\varepsilon}(x,\psi_{\varepsilon}(\xi))}{P_\Omega(x,\xi)} - 1\right| < \frac{\eta}{C P_\Omega(x,\xi)} \le \eta \ .
\end{equation*}
This completes the proof.
\end{proof}
Now we are in a position to prove Lemma \ref{max_princ}.
\begin{proof}[Proof of Lemma \ref{max_princ}]
Let $x \in \Omega$. Then for $\varepsilon>0$, sufficiently small, let $\Omega_{\varepsilon}:=\{y \in \Omega : d(y,\partial \Omega)> \varepsilon\}$ be a non-empty domain such that 
\begin{itemize}
\item $\overline{\Omega} \setminus \overline{\Omega_{\varepsilon}}$ is a tubular neighborhood of $\partial \Omega$,
\item $x \in \Omega_{\varepsilon}$ with $d(x,\partial \Omega_{\varepsilon}) > \frac{1}{2} d(x,\partial \Omega)$.
\end{itemize}
Let $\psi_{\varepsilon}$ denote the radial projection map from $\partial \Omega$ onto $\partial \Omega_{\varepsilon}$ given by,
\begin{eqnarray*}
\psi_{\varepsilon} : \partial \Omega &\to & \partial \Omega_{\varepsilon} \\
(\xi,0) &\mapsto & (\xi,\varepsilon) \:.
\end{eqnarray*}

Now by the sub-mean value property,
\begin{equation} \label{max_princ_eq1}
\varphi(x) \le \int_{\partial \Omega_{\varepsilon}} \varphi(\xi,\varepsilon)\:d\mu_{x,\Omega_{\varepsilon}}(\xi,\varepsilon)=\int_{\partial \Omega} \varphi(\psi_{\varepsilon}(\xi))\:\left(\frac{d(\psi^*_{\varepsilon} \mu_{x,\Omega_{\varepsilon}})}{d\mu_{x,\Omega}}\right)(\xi)\:d\mu_{x,\Omega}(\xi)\:.
\end{equation}
Then we note that
\begin{equation*}
\frac{d(\psi^*_{\varepsilon} \mu_{x,\Omega_{\varepsilon}})}{d\mu_{x,\Omega}}(\xi)=\left(\frac{P_{\Omega_\varepsilon}(x,\psi_{\varepsilon}(\xi))}{P_\Omega(x,\xi)}\right)\left(\frac{\psi^*_{\varepsilon}\left(d\Sigma_{\partial \Omega_\varepsilon}\right)}{d\Sigma_{\partial \Omega}}\right)(\xi)\:,
\end{equation*}
where  $\Sigma_{\partial \Omega}$ and $\Sigma_{\partial \Omega_\varepsilon}$ denote the normalized surface volume measures on $\partial \Omega$ and $\partial \Omega_{\varepsilon}$ respectively.

\medskip

Now we take an orthonormal basis $\{e_1,\dots,e_{n-1}\}$ of $T_{\xi}\partial \Omega$, with respect to the induced Riemannian metric on $\partial \Omega$. Then as $\psi_{\varepsilon} \to id_{\partial \Omega}$ in $C^\infty$ as $\varepsilon \to 0$, we have for all $1 \le i,j \le n-1$,
\begin{equation*}
\left\langle (d\psi_{\varepsilon})_\xi(e_i),(d\psi_{\varepsilon})_\xi(e_j)\right\rangle \to \langle e_i,e_j \rangle \text{ as } \varepsilon \to 0\:.
\end{equation*}
Hence,
\begin{equation*} 
\frac{(\psi^*_{\varepsilon}\left(d\Sigma_{\partial \Omega_\varepsilon}\right))_\xi(e_1,\dots,e_{n-1})}{(d\Sigma_{\partial \Omega})_\xi(e_1,\dots,e_{n-1})}=\frac{(d\Sigma_{\partial \Omega_\varepsilon})_{\psi_\varepsilon(\xi)} \left((d\psi_{\varepsilon})_\xi(e_1),\dots,(d\psi_{\varepsilon})_\xi(e_{n-1})\right)}{(d\Sigma_{\partial \Omega})_\xi(e_1,\dots,e_{n-1})} \to 1 \text{ as } \varepsilon \to 0\:.
\end{equation*}

\medskip

Combining the above with Lemma \ref{Poisson_asymptotic_lemma}, it follows that for all $\xi \in \partial \Omega$,
\begin{equation*}
\frac{d(\psi^*_{\varepsilon} \mu_{x,\Omega_{\varepsilon}})}{d\mu_{x,\Omega}}(\xi)  \to 1 \text{ as } \varepsilon \to 0\:.
\end{equation*}

\medskip

Then applying the Dominated Convergence Theorem with respect to $\mu_{x,\Omega}$ on $\partial \Omega$ in (\ref{max_princ_eq1}), the hypothesis on normal limits of $\varphi$ yields  
\begin{equation*}
\varphi(x) \le \int_{\partial \Omega} f(\xi)\:d\mu_{x,\Omega}(\xi)=P[f](x)\:.
\end{equation*}
\end{proof}

\section{Proofs of the main results} 
\begin{proof}[Proof of Theorem \ref{domain_nt}]
Let $\xi \in \partial \Omega$ and $z$ be a point in $\Omega$ with  
\begin{equation} \label{domain_nt_pf_eq1}
d(\xi,z) \le \frac{\{(1+\alpha)-(3+\alpha)\kappa\}}{(1+\kappa)} d(z,\partial \Omega)\:.
\end{equation}
We note that by the hypothesis on $\kappa$, one has
\begin{equation} \label{domain_nt_pf_eq}
1 <  \frac{\{(1+\alpha)-(3+\alpha)\kappa\}}{(1+\kappa)} \le 1+ \alpha \:.
\end{equation}

Now by Lemma \ref{props_of_v}, there exists a ball $B(x,r)$ such that 
\begin{itemize}
\item $z \in B(x,r)\:,$
\item $r \le \rho_0(x)=\kappa d(x,\partial \Omega)$ and 
\item the harmonic extension $h$ of $u$ on $B(x,r)$ satisfies $h(z)=v(z)$.
\end{itemize}
We claim that 
\begin{equation} \label{domain_nt_pf_eq2}
\overline{B(x,r)} \subset \Gamma_\alpha(\xi)\:.
\end{equation}
To establish the claim, we first note that 
\begin{equation} \label{domain_nt_pf_eq3}
d(\xi,x) \le \{(1+\alpha)-(2+\alpha)\kappa\}d(x,\partial \Omega)\:.
\end{equation}
By repeated application of the triangle inequality, (\ref{domain_nt_pf_eq1}) and the fact that $r \le \kappa d(x,\partial \Omega)$, this is seen as follows:
\begin{eqnarray*}
d(\xi,x) & \le & d(\xi,z)+d(z,x) \\
         & \le & \frac{\{(1+\alpha)-(3+\alpha)\kappa\}}{(1+\kappa)}\: d(z,\partial \Omega) + \kappa d(x, \partial \Omega) \\
         & \le & \frac{\{(1+\alpha)-(3+\alpha)\kappa\}}{(1+\kappa)} \left(d(x,\partial \Omega)+d(z,x)\right) + \kappa d(x, \partial \Omega) \\
         & \le & \left[\frac{\{(1+\alpha)-(3+\alpha)\kappa\}}{(1+\kappa)} + \frac{\{(1+\alpha)-(3+\alpha)\kappa\}}{(1+\kappa)}\kappa + \kappa\right] d(x,\partial \Omega) \\
         & = & \{(1+\alpha)-(2+\alpha)\kappa\} d(x,\partial \Omega)\:.
\end{eqnarray*}
Then for any $y \in \overline{B(x,r)}$, by (\ref{domain_nt_pf_eq3}) and the fact that $r \le \kappa d(x,\partial \Omega)$, we get
\begin{eqnarray*}
d(\xi,y) & \le & d(\xi,x)+ r \\
         & \le & \{(1+\alpha)-(2+\alpha)\kappa\} d(x,\partial \Omega) + \kappa d(x,\partial \Omega) \\
         & = & (1+\alpha)(1-\kappa)d(x,\partial \Omega) \\
         & \le & (1+\alpha)\{d(x,\partial \Omega) - r\} \\
         & \le & (1+\alpha)\{d(x,\partial \Omega) - d(x,y)\} \\
         & \le & (1+\alpha)\:d(y,\partial \Omega) \:.
\end{eqnarray*}
This proves the claim (\ref{domain_nt_pf_eq2}). 

\medskip

Now let $\xi \in \partial \Omega$ be such that \begin{equation*}  
\displaystyle\lim_{z \to \xi} u(z) = u_\xi \:\text{exists} \:,
\end{equation*}
whenever $z$ approaches to $\xi$ through the non-tangential cone $\Gamma_\alpha(\xi)$. Let $\{z_n\}_{n=1}^\infty$ be a sequence of points in $\Omega$ that converges normally to $\xi$ (as in definition \ref{normal_convergence}). By Lemma \ref{props_of_v}, one has balls $B(x_n,r_n)$ such that 
\begin{itemize}
\item $z_n \in B(x_n,r_n)\:,$
\item $r_n \le \rho_0(x_n)=\kappa d(x_n,\partial \Omega)$ and 
\item the harmonic extension $h_n$ of $u$ on $B(x_n,r_n)$ satisfies $h_n(z_n)=v(z_n)$.
\end{itemize}
Then by (\ref{domain_nt_pf_eq1}), (\ref{domain_nt_pf_eq}) and (\ref{domain_nt_pf_eq2}), it follows that for all $n$ sufficiently large,
\begin{equation*} 
\overline{B(x_n,r_n)} \subset \Gamma_\alpha(\xi)\:.
\end{equation*}
Thus,
\begin{eqnarray} \label{proving_limit_of_v}
|v(z_n)-u_\xi| & \le & \int_{\partial B(x_n,r_n)} |u(y)-u_\xi|\:d\mu_{z_n,B(x_n,r_n)}(y) \nonumber \\
& \le & \displaystyle\sup_{y \in \overline{B(x_n,r_n)}} |u(y)-u_\xi| \to 0 \text{ as } n \to \infty \:.
\end{eqnarray}
Hence,
\begin{equation*}
\displaystyle\lim_{z \to \xi} v(z) = u_\xi \:,
\end{equation*}
whenever $z$ converges normally to $\xi$, for almost every boundary point $\xi$, with respect to the Riemannian measure on $\partial \Omega$ and thus also with respect to the harmonic measures on $\partial \Omega$. 

\medskip

Using remark \ref{props_of_w} and proceeding as above, one also gets that
\begin{equation*}
\displaystyle\lim_{z \to \xi} w(z) = u_\xi \:,
\end{equation*}
whenever $z$ converges normally to $\xi$, for almost every boundary point $\xi$, with respect to the harmonic measures on $\partial \Omega$. Combining the above with Lemma \ref{props_of_v} and remark \ref{props_of_w}, we see that $\varphi:=v-w$ is a non-negative, bounded, continuous, subharmonic function so that
\begin{equation*}
\displaystyle\lim_{z \to \xi} \varphi(z) = 0 \:,
\end{equation*}
whenever $z$ converges normally to $\xi$, for almost every boundary point $\xi$, with respect to the harmonic measures on $\partial \Omega$. An application of Lemma \ref{max_princ} then implies that $\varphi \equiv 0$, that is, $u \equiv v \equiv w$. So $u$ is both subharmonic as well as superharmonic, and hence $u$ is harmonic.
\end{proof}

\begin{proof}[Proof of Theorem \ref{harmonic_manifold_nt}]
Let $\xi \in \partial X$ be such that 
\begin{equation*}
\displaystyle\lim_{z \to \xi} u(z) = u_\xi \:\text{exists}\:,
\end{equation*}
whenever $z$ approaches to $\xi$ through the non-tangential cone $T_C(\xi)$. Let $\{z_n\}_{n=1}^\infty$ be a sequence of points lying on the geodesic ray $\gamma_\xi$, such that $z_n \to \xi$ as $n \to \infty$.

\medskip

Then by Lemma \ref{props_of_v}, there exist balls $B(x_n,r_n)$ such that 
\begin{itemize}
\item $z_n \in B(x_n,r_n)\:,$
\item $r_n \le \rho_0(x_n)\equiv C/2$, and 
\item the harmonic extension $h_n$ of $u$ on $B(x_n,r_n)$ satisfies $h_n(z_n)=v(z_n)$.
\end{itemize}
Next we note that for all $n \in \N$,
\begin{equation*}
\overline{B(x_n,r_n)} \subset T_C(\xi)\:.
\end{equation*}
This is seen as follows. For $y \in \overline{B(x_n,r_n)}$,
\begin{equation*}
d(y,\gamma_\xi) \le d(y,z_n) \le d(y,x_n)+d(x_n,z_n) \le 2r_n \le C\:.
\end{equation*}

Then by proceeding as in the proof of Theorem \ref{domain_nt}, we see that $\varphi:=v-w$ is a non-negative, bounded, continuous, subharmonic function so that
\begin{equation} \label{harmonic_manifold_nt_pf_eq1}
\displaystyle\lim_{z \to \xi} \varphi(z) = 0 \:,
\end{equation}
whenever $z$ approaches to $\xi$ along the geodesic ray $\gamma_\xi$, for almost every boundary point $\xi$, with respect to the visibility measure $\lambda_o$ on $\partial X$.

\medskip

Now boundedness of $\varphi$ implies that $\varphi$ has a harmonic majorant. Then Lemma \ref{least_harmonic_majorant} yields that $\varphi$ has a unique non-negative least harmonic majorant $F_\varphi$. The boundedness of $\varphi$ then implies that the Dominated Convergence Theorem is applicable and which in turn yields that
\begin{eqnarray*}
F_\varphi(o)&=& \displaystyle\lim_{r \to \infty} \int_{T^1_o X} \varphi(\gamma_{o,v}(r))\:d\theta_o(v) \\
&=&\displaystyle \int_{T^1_o X} \left(\lim_{r \to \infty} \varphi(\gamma_{o,v}(r))\right)\:d\theta_o(v) \\
&=&0\:.
\end{eqnarray*} 
The last line follows from  (\ref{harmonic_manifold_nt_pf_eq1}) and the fact that the visibility measure $\lambda_o$ is the push-forward of $\theta_o$ (on $T^1_o X$) to the boundary $\partial X$ under the radial projection. 

\medskip

Then by the maximum principle, $F_\varphi \equiv 0$. Hence,
\begin{equation*}
0 \le \varphi \le F_\varphi \equiv 0 \:.
\end{equation*}
Thus $\varphi \equiv 0$ and this completes the proof.  
\end{proof}

\section*{Acknowledgements} The author would  like to thank Prof. Kingshook Biswas for many useful discussions. The author is thankful to him and Prof. Swagato K. Ray for suggesting the problems. The author is supported by a Research Fellowship of Indian Statistical Institute. 

\bibliographystyle{amsplain}

\end{document}